\documentclass{amsart}

\usepackage{amssymb,amsfonts, amsmath, amsthm}
\usepackage[all,arc]{xy}
\usepackage{enumerate}
\usepackage{mathrsfs}
\usepackage{mathtools}
\usepackage[mathscr]{euscript}
\usepackage{array}
\usepackage{tikz}
\usepackage{tikz-cd}
\usepackage{mathpazo} 
\usepackage{textcomp}
\usepackage{bbold}
\usepackage[bbgreekl]{mathbbol}
\usepackage{enumitem}

\usepackage[pagebackref, colorlinks=true, linkcolor=blue, citecolor = red]{hyperref}
\renewcommand*{\backref}[1]{}
\renewcommand*{\backrefalt}[4]{({%
		\ifcase #1 Not cited.%
		\or On p.~#2%
		\else On pp.~#2%
		\fi%
	})}

\usepackage[nameinlink,capitalise,noabbrev]{cleveref}
\crefname{subsection}{Subsection}{Subsection}

\usetikzlibrary{patterns}

\usepackage{geometry}
\DeclareMathAlphabet{\mathbbe}{U}{bbold}{m}{n}

\def\DDelta{{\mathbbe{\Delta}}}
\newcommand{\DD}{\DDelta}

\renewcommand{\2}{\mathbbe{2}}

\newcommand{\C}{\mathscr{C}}
\newcommand{\cC}{\mathcal{C}}
\newcommand{\D}{\mathscr{D}}
\newcommand{\cF}{\mathcal{F}}
\newcommand{\bN}{\mathbb{N}}
\newcommand{\cM}{\mathcal{M}}
\newcommand{\kT}{\mathfrak{T}}
\newcommand{\V}{\mathscr{V}}
\newcommand{\cW}{\mathcal{W}}
\newcommand{\cV}{\mathcal{V}}
\newcommand{\Y}{\mathscr{Y}}

\newcommand{\set}{\mathscr{S}\mathrm{et}}
\newcommand{\sset}{s\set}
\newcommand{\Sub}{\mathscr{S}\mathrm{ub}}
\newcommand{\Fun}{\mathrm{Fun}}
\newcommand{\Map}{\mathrm{Map}}
\newcommand{\Hom}{\mathrm{Hom}}
\newcommand{\Ho}{\mathrm{Ho}}
\newcommand{\Mono}{\mathcal{M}\mathrm{ono}}

\newtheorem{theorem}[equation]{Theorem}
\newtheorem{lemma}[equation]{Lemma}
\newtheorem{proposition}[equation]{Proposition}
\newtheorem{corollary}[equation]{Corollary}

\theoremstyle{definition}
\newtheorem{definition}[equation]{Definition}
\newtheorem{example}[equation]{Example}

\theoremstyle{remark}
\newtheorem{remark}[equation]{Remark}
\newtheorem{notation}[equation]{Notation}

\numberwithin{equation}{section}

\title{Simplicial Homotopy Type Theory is not just Simplicial: What are $\infty$-Categories?}

\date{August 2025}

\author{Nima Rasekh}
\address{Institut f{\"u}r Mathematik und Informatik, Universit{\"a}t Greifswald, Greifswald, Germany}
\email{nima.rasekh@uni-greifswald.de}


\begin{document}


\begin{abstract}
 $\infty$-category theory was originally developed in the context of classical homotopy theory using standard set theoretical assumptions, but has since been extended to a variety of mathematical foundations. One such successful effort, primarily due to Martini and Wolf, introduced a theory of $\infty$-categories internal to the foundation of an arbitrary \emph{Grothendieck $\infty$-topos}, meaning they used categorical foundations \cite{martiniwolf2024colimits}. Another approach, due to Riehl and Shulman, developed a theory of $\infty$-categories internal to their own type theory: \emph{simplicial homotopy type theory (sHoTT)}, meaning they employed a (homotopy) type theoretic foundation  \cite{riehlshulman2017rezktypes}. 

	One aspect of developing a theory of $\infty$-categories in different foundations consists of introducing ways to translate from one foundation to another. Concretely, as part of their work, Riehl and Shulman prove that $\infty$-categories internal to Grothendieck $\infty$-topoi give us categorical models of sHoTT. In fact the name ``simplicial'' in sHoTT suggests that all categorical models of sHoTT should be given by simplicial objects in suitable $\infty$-categories. In this paper we prove that contrary to this expectation, there are models of sHoTT that are not simply simplicial objects. This suggests that in a general foundations, the notion of $\infty$-category is more general than previously assumed.
\end{abstract}

\maketitle

\section{Introduction}

\subsection{From Homotopy Theory to \texorpdfstring{$\infty$}{oo}-Category Theory} 
Category theory is a powerful meta-mathematical formalism introduced by Eilenberg and Mac Lane \cite{eilenbergmaclane1945categories} to compare and contrast different mathematical theories. Different theories, such as sets, groups, topological spaces, manifolds, ... assemble into different categories, those theories can then be compared via functors. 

Unfortunately, not all mathematical theories of interest naturally assemble into the strict formalism of a category. For example, in algebraic topology or homological algebra, we often want to analyze weak equivalences between topological spaces or quasi-isomorphisms between chain complexes, none of which can be recovered out of the categories of topological spaces or chain complexes. Similarly, while classical algebraic operations (groups, rings, modules, ...) naturally fit into categories, certain algebraic operation necessitate further structure, such as the concatenation of paths or the tensor product of rings, both of which are only associative up to suitable homotopies or isomorphisms. Motivated by these first examples, theories that required such additional structure or axioms were named \emph{homotopical theories} or \emph{homotopy theories}.

As a result, starting from the 1960s, a quest commenced to develop a theory of \emph{$\infty$-categories}: a generalization of category theory, that includes categories, but is also able to incorporate these new generalized examples and in general provide a framework for \emph{homotopical mathematics}. One first contender was the notion of a \emph{model category}, due to Quillen \cite{quillen1967modelcats}. By complementing a category with additional classes of morphisms, and in particular an axiomatic class of weak equivalences, he was able to recover examples of interest, and in particular weak equivalences of topological spaces and quasi-isomorphisms of chain complexes. In the past 60 years, model category theory has significantly contributed to the study and advancement of homotopical mathematics.

Of course model categories could not remain the only contender for a theory of $\infty$-categories. Indeed, the very rigid definition of model categories, while beneficial for certain applications, such as computing limits, makes them wholly unsuitable for other applications, such as the construction of functor categories between model categories. As a result, new formalisms have been proposed that can better handle these formal aspects of $\infty$-category theory. 

A first such contender was the notion of \emph{quasi-categories}. While first defined by Boardman and Vogt \cite{boardmanvogt1973qcats}, a first serious development of $\infty$-category via quasi-categories was due to Joyal in the 90s \cite{joyal2008theory,joyal2008notes} and later continued by Lurie \cite{lurie2009htt}. Around the same time, alternative definitions to quasi-categories emerged, such as a theory of \emph{complete Segal spaces} due to Rezk \cite{rezk2001css}, and a theory of \emph{Kan-enriched categories}, due to Bergner \cite{bergner2007bergnermodelcat}, all three of which (and many more) have been proven to be equivalent to each other in a suitable sense \cite{joyaltierney2007qcatvssegal,bergner2007threemodels}. As a result of these developments, in the last 20 years $\infty$-category theory has established itself as the foundational theory for homotopical mathematics, with applications ranging from algebraic topology and derived geometry to mathematical physics.

\subsection{Internalizing \texorpdfstring{$\infty$}{oo}-Category Theory: Categories and Type Theories}
In the previous subsection, we discussed how the rise of (classical) category theory has significantly contributed to the study of a variety of mathematical theories, however, did not tackle in which setting categories are themselves defined. Fortunately, categories involve a finite collection of data (a collection of objects, a collection of morphisms, and a composition operation, that is associative and unital). Hence, categories have successfully been defined in any arbitrary mathematical foundation, ranging from a classical set theoretic foundation, to the more syntactical type theoretical foundation, and even in a categorical foundation, via \emph{internal categories}. This means, similar to group theory or most other algebraic theories, category theory can be studied independently of any particular choice of foundation.

Unfortunately, every known definition of an $\infty$-category involves an infinite tower of data and axioms. As a result first definitions of $\infty$-categories (the ones mentioned above) are explicitly tied to specific foundational choices (meaning we start with one specific category, the category \emph{simplicial sets}, and use that to define $\infty$-categories), and it has remained an open question how we can do $\infty$-category theory in an arbitrary foundation, and what even a suitable foundation for $\infty$-category theory should be. Beyond purely theoretical considerations, a proper development of $\infty$-category theory in arbitrary foundations also has important practical applications. Indeed, we are now witnessing the rise of \emph{formalization of mathematics} via proof assistants as a way to both formally verify research mathematics as well as use software and AI tools to advance mathematics. However, a proper application of formalization significantly benefits from a suitable type theoretic foundation. Hence, a proper formalization of $\infty$-category theory has as an important building block a type theory that internalizes $\infty$-categories.

As a result of these challenges, a variety of approaches have been proposed to develop $\infty$-category theory in a variety of foundations. One approach, using internal $\infty$-categories, generalizes Rezk's definition of complete Segal spaces \cite{rezk2001css} to an arbitrary $\infty$-categorical foundation \cite{rasekh2022cartesian}. This has resulted in a significant advancement in our understanding of internal $\infty$-categories, primarily due to the work of Martini and Wolf, who study $\infty$-categories internal to the foundation of Grothendieck $\infty$-topoi \cite{martiniwolf2024colimits}. On the other side, Riehl and Shulman proposed a type theoretic way to develop $\infty$-categories internal to a type theories, by introducing \emph{simplicial Homotopy Type Theory (sHoTT)}, which is a foundation in which $\infty$-categories occur axiomatically. Moreover, this type theory has lead to the development of the proof assistant \emph{Rzk}, which has successfully formalized various $\infty$-categorical results \cite{kudasovriehlweinberger2024yoneda}. More recently, we are witnessing the advancement of a project spearheaded by Cisinski, Cnossen, Nguyen, and Walde, providing a further foundational method to study and develop $\infty$-category theory \cite{cisinskicnossennguyenwalde2025higher}.

\subsection{Constructing Models: From Type Theory to Category Theory} \label{subsec:constructing models}
The recent rise of $\infty$-category theory internal to different foundations is a promising development. However, many questions have remained unsettled. Indeed, in the classical situation, there are precise ways to translate between type theoretic foundations (syntax) and categorical foundations (semantics), via what is known as \emph{syntax-semantics duality}. In that context these categories are known as \emph{models of type theories} and the type theories provide the \emph{internal language of categories}. In particular, it is established that a category internal to a categorical foundation translates suitably to a category in a type-theoretic foundation \cite[Section D.4]{johnstone2002elephantsii}, \cite{seely1983hyperdoctrines,seely1984locallycartesian,lambekscott1988higherorderlogic,hofmann1995lccc,curiengarnerhofmann2014categorical,clairambaultdybjer2014biequivalence}. 

We would similarly anticipate a comparison operation that allows translating from a type theoretical foundation (the syntax) to a categorical foundation (semantics). As a first step, as part of their work, Riehl and Shulman prove that that every categorical foundation considered by Martini-Wolf, i.e.~$\infty$-categories internal to Grothendieck $\infty$-topoi, are in particular models of simplicial homotopy type theory, in part motivated by the fact that Grothendieck $\infty$-topoi model homotopy type theory \cite{shulman2019inftytoposunivalent}. Ideally we would have expected that $\infty$-categories internal to a broader class of ambient $\infty$-categories will recover all models, giving us a precise correspondence.

At this point it is instructive to take a step back and provide a more detailed background. Homotopy type theory is a foundation for mathematics consisting of types (our foundational building blocks, analogous to sets) and terms (that live in types, analogous to elements) and various type constructors and axioms, one of which is the \emph{univalence axiom} that guarantees the types exhibit homotopical behavior \cite{hottbook2013}. Riehl and Shulman start with such a homotopy type theory and add a strict directed interval $\2$, which has two unequal terms $0_{\2},1_{\2}$, resulting in simplicial homotopy type theory. Using this directed interval, they derive all kinds of simplicial shapes in their type theory, such as higher dimensional cubes, or higher dimensional tetrahedra. Riehl and Shulman then use the properties of these axiomatic shapes to internally define a \emph{Segal condition} and a \emph{completeness condition}, which allows them to define a notion of complete Segal space internal to simplicial homotopy type theory, which they call a \emph{Rezk type} and provides their notion of $\infty$-category.

As part of their work, Riehl and Shulman also construct explicit categorical models for their type theory. Given the construction outlined above, it is unsurprising their choice of model is a category of the form $\cM^{\DD^{op}}$, where $\cM$ is a suitable model category and $\DD$ the simplex category. In that setting $\2$ corresponds to the free walking arrow $\Delta[1]$, and we can then similarly translate the Segal and completeness condition, meaning the $\infty$-categories in sHoTT, the Rezk types, correspond precisely to complete Segal objects internal to $\cM$. Given this result, the expectation might have been that \textbf{every} model of sHoTT is given by simplicial objects in a suitable category.

\subsection{What is an \texorpdfstring{$\infty$}{oo}-Category?}
The aim of this paper is to prove that this is indeed \textbf{not} the case, and that the theory of $\infty$-categories in sHoTT developed by Riehl and Shulman admits categorical models that are not simplicial objects in some category, meaning this theory transcends all known notions of internal $\infty$-categories. This is the content of \cref{thm:main nonsimplicial}. This result is indeed a surprise and stands in stark contrast to the situation for classical categories, where categories internal to type theories can be in fact be recovered inside categorical models via a naive notion of internal category. It is hence instructive to understand the intuition behind this result. 

At a fundamental level, a category is characterized via \emph{four} pieces of data: objects, morphisms, two composable morphisms, and three composable morphisms. The fact that we only need four objects (or types) means that this definition is independent of any choice of foundation. Any definition of $\infty$-category, however, requires an infinite collection of objects (or types), and, more precisely, an $\bN$-indexed collection, giving us the $n$-morphisms for all $n$. Unlike any particular choice of finite number, the collection of natural numbers does in fact explicitly depend on the foundation we have chosen. This idea is the essence of the counter-example, where we use the filter quotient construction to construct a categorical model in which there is an uncountable (seen externally) collection of natural numbers, necessitating additional simplices. This means \textbf{what} an $\infty$-category is fundamentally depends on the foundation in which we are working. See \cref{rem:too many simplices} for more details.

Finally, let us note that beyond theoretical implications, there are also concrete mathematical implications. Namely, this result demonstrates that theorems proven in the language of sHoTT, such as \cite{buchholtzweinberger2023synthetic,weinberger2024twosided}, are in fact stronger than previously anticipated, as they apply to a much broader class of examples. For example, even though the nature of the $\infty$-category in \cref{ex:filter product} might appear unfamiliar, given that it is a model of sHoTT, it is still the case that it satisfies the Yoneda lemma, as proven in \cite[Theorem 9.1]{riehlshulman2017rezktypes}.

\subsection{Filter Quotient Construction}
The construction of the example employs the \emph{filter quotient construction}. A filter is defined as a subset of a poset, often taken to be the power set. The idea of using filters in mathematical logic to construct non-trivial models goes back to major figures such as Skolem \cite{skolem1934earlyultraproducts} and {\L}o{\'s} \cite{los1955ultraproduct}. Additionally, while not explicitly in the proof, an intuition regarding filter quotients plays an important role in Cohen's forcing construction, proving the independence of the continuum hypothesis \cite{cohen1963forcing}. This motivated Lawvere and Tierney to introduce topos theory as a categorical foundation for mathematics and define therein \emph{filter quotients} as a way to construct new models \cite{lawvere1964elementarysets,tierney1972elementarycontinuum}. These techniques were further refined by Adelman and Johnstone \cite{adelmanjohnstone1982serreclasses}, and have since become a standard part of modern topos theory \cite{johnstone1977topostheory,maclanemoerdijk1994topos}. Recently, the filter quotient construction has also been generalized to $\infty$-categories \cite{rasekh2021filterquotient} and model categories \cite{rasekh2025filtermodelcat}. In this paper we build on these developments and prove that the filter quotient construction preserves models of sHoTT and use that to construct non-trivial models of interest. See \cref{cor:filter quotient model} for more details.

\subsection{Acknowledgment}
I would like to thank Emily Riehl for helpful discussions and comments. I also want to thank Jonathan Weinberger for his careful reading and helpful comments on the first draft of this paper. I am also grateful to the Max Planck Institute for Mathematics in Bonn for its hospitality and financial support. Moreover, I am grateful to the Hausdorff Research Institute for Mathematics in Bonn, Germany, for organizing the trimester ``Prospects of Formal Mathematics,'' funded by the Deutsche Forschungsgemeinschaft (DFG, German Research Foundation) under Germany's Excellence Strategy – EXC-2047/1 – 390685813, which resulted in many fruitful interactions and relevant conversations regarding foundation and type theory.

\section{Model Categories and Filter Quotients} \label{sec:model cat}
Let us review some basic facts regarding model categories and filter quotients. 

\begin{definition} \label{def:model category}
	A model structure on a category $\C$ is a triple $(\cF,\cC,\cW)$ of classes of morphisms in $\cM$ satisfying the following axioms:
	\begin{itemize}[leftmargin=*]
		\item $(\cC \cap \cW, \cF)$ and $(\cC, \cF \cap \cW)$ are weak factorization systems.
		\item If two of $f$, $g$, and $g \circ f$ are in $\cW$, so is the third.
	\end{itemize}
	A model category is a tuple $(\C,\cM)$ of a (co)complete category $\C$ and a model structure $\cM$. If $\C$ is only finitely (co)complete, we call it a \emph{finitary model category}. 
\end{definition}

We now review how to construct new model structures via filter quotients.

\begin{definition}
	Let $(P, \leq)$ be a poset. A \emph{filter} on $P$ is a non-empty subset of $P$, which is:
	\begin{itemize}[leftmargin=*]
		\item Upwards closed: for all $x \in \Phi$ and $y \geq x$, $y \in \Phi$.
		\item Intersection closed: for all $x,y$ in $\Phi$, there exists a $z$ in $\Phi$, with $z \leq x, y$.
	\end{itemize}
\end{definition}

Recall that an object $U$ in a category $\C$ is subterminal, if for every object $X$ in $\C$, there is at most one morphism from $X$ into $U$. 

\begin{definition}
	Let $\C$ be a category. A \emph{filter of subterminal objects} on $\C$ is a filter on the poset of subterminal objects $\Sub(\C)$.
\end{definition}

\begin{definition} \label{def:filter quotient category}
 Let $\C$ be a category with finite products and $\Phi$ a filter of subterminal objects. Then the filter quotient $\C_\Phi$ is a category with the same objects as $\C$, and for two objects $X, Y$ 

 \[\Hom_{\C_\Phi}(X,Y)=\left(\coprod_{U\in\Phi}\Hom_\C(X\times U,Y)\right)\mathbin{\big/\mkern-6mu\sim_\Phi}\]
	where $f\colon U \times X \to Y$, $g\colon V \times X \to Y$ are equivalent if there exists a $W \in \Phi$, with $W \leq U, V$, and $f$ and $g$ are equal when restricted to $W$.
\end{definition}

There is an evident \emph{projection functor} $P_\Phi\colon \C \to \C_\Phi$, which is the identity on objects and sends each morphism to its equivalence class, and preserves many properties of interest. Recall the following basic facts from \cite[Example D.5.1.7]{johnstone2002elephanti}. 

\begin{proposition} \label{prop:filter quotient category}
	Let $\C$ be a category with finite products, $\Phi$ a filter of subterminal objects.
	\begin{enumerate}[leftmargin=*]
		\item $P_\Phi$ preserves finite (co)limits, monomorphisms, (local) exponentiability, and subobject classifiers.
		\item If $\C$ is (locally) Cartesian closed, or an elementary topos, then so is $\C_\Phi$.
	\end{enumerate} 
\end{proposition}

In general this definition does not interact well with model structures, necessitating adjusting the definition. Following \cite[Definition 3.9]{rasekh2025filtermodelcat}, for a given model category $\cM$ and filter of subterminal objects $\Phi$, a class of morphisms $S$ in $\cM$ is $\Phi$-product stable, if for every $f$ in $S$ and $U$ in $\Phi$, $f \times U$ is in $S$.

\begin{definition} \label{def:model filter}
	Let $\cM$ be a finitary model category. A \emph{model filter} $\Phi$ on $\cM$ is a filter of subterminal objects with the following properties:
	\begin{enumerate}[leftmargin=*]
		\item Every $U$ in $\Phi$ is fibrant.
		\item The cofibrations and weak equivalences in $\cM$ are $\Phi$-product stable.
		\end{enumerate}
\end{definition}

With this definition we do have the expected result. 

\begin{notation}
	Let $\C$ be a category with finite products and $\Phi$ a filter of subterminal objects. Let $S$ be a $\Phi$-product stable set of morphisms. Let $S_\Phi$ be the set of morphisms in $\C_\Phi$ with the property that $f \in S_\Phi$ if and only if there exists $U \in \Phi$ such that $f \times U \in S$.
\end{notation}

\begin{theorem}[{\cite[Theorem 3.17, Proposition 4.3]{rasekh2025filtermodelcat}}] \label{thm:filter quotient model structure}
	Let $\cM$ be a finitary model category and $\Phi$ a model filter on $\cM$. The filter quotient $\cM_\Phi$ carries a model structure given by $(\cF_\Phi,\cC_\Phi,\cW_\Phi)$. In particular, $P_\Phi\colon\cM \to \cM_\Phi$ preserves fibrations, cofibrations, weak equivalences, and right properness.
\end{theorem}

Finally, let us observe one particularly relevant example: filter products.

\begin{example}[{\cite[Definition 6.4]{rasekh2025filtermodelcat}}] \label{ex:filter product}
	Let $\cM$ be a finitary model category with strict initial object and $I$ a set and $\Phi$ a filter of subsets of $I$. The \emph{filter product} $\prod_\Phi \cM$ is the filter quotient of $\prod_I\cM$, denoted $\prod_\Phi\cM$. It carries the filter product model structure, with the same properties as in \cref{thm:filter quotient model structure}.
\end{example}

\section{Non-Standard Models of sHoTT: sHoTT is not just simplicial}
As mentioned in \cref{subsec:constructing models}, Riehl and Shulman introduce simplicial homotopy type theory \emph{(sHoTT)} as an extension of homotopy type theory that incorporates a directed interval \cite{riehlshulman2017rezktypes}. In order to provide a semantics for their definition, they also define model categories with $\kT$-shapes, as a right proper cofibrantly generated model category with suitable data, modeling the relevant information. See \cite[Definition A.5]{riehlshulman2017rezktypes} for more details. They then prove that if the coherent theory $\kT$ has a strict interval \cite[3.1]{riehlshulman2017rezktypes}, such a model category models sHoTT.

Ideally, we would have liked to simply prove that the filter quotient construction preserves model categories with $\kT$-shapes and the strict interval in $\kT$, but unfortunately, this is not the case, as it does not preserve local presentability or cofibrant generation. As a first step, we hence define a finite analogue, called a \emph{finitary model category with $\kT$-shapes} (\cref{def:finitary model}), which strictly generalizes the original definition (\cref{lemma:finite vs. infinite}). We now have the following main result, justifying this definition.

\begin{theorem} \label{thm:main model}
	Let $\kT$ be a coherent theory with strict interval. Then finitary model categories with $\kT$-shapes model sHoTT.
\end{theorem}

See \cref{thm:main model precise}	for a more precise formulation. Now that we have this generalized definition, we can demonstrate it interacts well with filter quotients.

\begin{theorem} \label{thm:main filter quotient}
	Finitary model categories with $\kT$-shapes and strict intervals are preserved by filter quotients.
\end{theorem}

See \cref{thm:main filter quotient precise} for a more precise formulation. 
We can in particular apply this to filter products.

\begin{corollary} \label{cor:filter quotient model}
	Let $(\cM,\cV, \kT, \overline{\omega})$ be a finitary model category with $\kT$-shapes, such that $\kT$ has a strict interval, $I$ a set and $\Phi$ a filter of subsets of $I$, then the filter product $\prod_\Phi \cM$ models simplicial homotopy type theory.
\end{corollary}

Using the fact that simplicial objects in nice enough model categories are in particular model categories with $\kT$-shapes \cite[Example A.14]{riehlshulman2017rezktypes}, we also have the following corollary.

\begin{corollary}
		Let $\cM$ be a right proper model category, such that cofibrations are monomorphisms and the underlying category is an elementary topos (for example a right proper Cisinski model category). Let $I$ be a set and $\Phi$ a filter of subsets of $I$, then the filter product $\prod_\Phi (\cM^{\DD^{op}})$ models simplicial homotopy type theory.
\end{corollary}

Let us now apply this to concrete examples. 

\begin{example}
	Let $\sset^{\DD^{op}}$ denote the category of bisimplicial sets. The injective model structure is right proper with cofibrations precisely the monomorphisms. So, for a set $I$ and a filter of subsets $\Phi$, $\prod_\Phi \sset^{\DD^{op}}$ models simplicial homotopy type theory.
\end{example}

With this last example at hand, we can transition to the second main result of the paper, namely that simplicial homotopy type theory has models which are not just simplicial objects in a category. Recall that for a simplicial model category $\cM$, the underlying $\infty$-category is simply the full Kan-enriched subcategory of bifibrant objects in $\cM$. We will denote this $\infty$-category by $\Ho_\infty(\cM)$.

The injective model structure on $\sset^{\DD^{op}}$ is indeed simplicial, and for every set $I$ and filter of subsets $\Phi$, the filter product $\prod_\Phi \sset^{\DD^{op}}$ inherits the structure of a simplicial model category \cite[Corollary 6.10]{rasekh2025filtermodelcat}. Thus we can state the following result. 

\begin{theorem} \label{thm:main nonsimplicial}
	Let $\sset^{\DD^{op}}$ be the injective model structure on bisimplicial sets and $\cF$ a non-principal filter on $\bN$. Then the following hold:
	\begin{itemize}[leftmargin=*]
		\item $\prod_\cF \sset^{\DD^{op}}$ is a model of simplicial homotopy type theory.
		\item For every $\infty$-category $\C$, there does not exist an equivalence  $\Ho_\infty(\prod_\cF \sset^{\DD^{op}}) \simeq \C^{\DD^{op}}$.
	\end{itemize}
\end{theorem}

Let us provide some intuition behind this result. For that we focus on one particular example, namely the Fr{\'e}chet filter $\cF$ on the natural numbers $\bN$, which by definition includes subsets of $\bN$ whose complement is finite. Unwinding \cref{def:filter quotient category}, an object in $\prod_\cF \sset^{\DD^{op}}$ is a tuple $(X_n)_{n \in \bN}$, where $X_n$ is an object in $\sset^{\DD^{op}}$, and a morphism is of the form  $(f_n)_{n \in \bN}\colon (X_n)_{n \in \bN} \to (Y_n)_{n \in \bN}$, where two morphism $(f_n)_{n \in \bN}$ and $(g_n)_{n \in \bN}$ are in the same equivalence class if there exists an $N \in \bN$, such that for all $n > N$, $f_n = g_n$, meaning they are ``eventually equal". See also \cite[Section 3.2]{rasekh2021filterquotient} for a more detailed description.

\begin{remark} \label{rem:too many simplices}
	Under minor conditions on the $\infty$-category $\C$, there is a well-defined Yoneda functor $\DD \to \C^{\DD^{op}}$, which induces a conservative functor $\C^{\DD^{op}} \to \prod_\bN \C$, concretely given by evaluating at the $\Delta[n]$. Hence, as long as $\C$ is nice enough (which will be the case in our applications), the behavior of an arbitrary simplicial object in $\C$ is controlled by a countable choice of objects (the $\Delta[n]$).
	
	Interestingly enough, the $\infty$-category $\Ho_\infty(\prod_\cF \sset^{\DD^{op}})$ does not behave this way. Here, it is instructive to carefully analyze its internal logic. The natural number object in $\Ho_\infty(\prod_\cF \sset^{\DD^{op}})$ is given by the constant tuple $(\bN)_{n \in \bN}$ and the terminal object by $(1)_{n\in \bN}$. By definition an ``internal natural number'' is a morphism $(1)_{n \in \bN} \to (\bN)_{n \in \bN}$, which, as we explained right above, corresponds to a collection of morphisms $1 \to \bN$, or just a natural sequence	$(a_n)_{n \in \bN}$, up to eventual equality. This means the set of internal natural numbers is the set of natural sequences up to eventual equality, which is an (externally) uncountable set. 
	
	Now, for every internal natural number $(a_n)_{n \in \bN}$ we obtain an ``internal simplex'' $\Delta[(a_n)_{n \in \bN}]$. Meaning following the argumentation in the first paragraph, we will get a conservative functor 
	\[\Ho_\infty(\prod_\cF \sset^{\DD^{op}}) \to \prod_{(a_n)_{n \in \bN} \in \Hom_{\prod_\cF \sset^{\DD^{op}}}((1)_{n \in \bN},(\bN)_{n \in \bN})} \Ho_\infty(\prod_\cF \sset).\] 
	However, as we discussed above, the number these equivalence classes of sequences is uncountable, meaning we cannot ``pick'' a countable collection of objects, i.e. generators, therein that could be the image of $\DD$, implying that there can never be an equivalence to an $\infty$-category of simplicial objects.
\end{remark}

\section{Proof: Filter Quotients Model Simplicial Homotopy Type Theory} \label{sec:filter quotient}
This section is dedicated to precise formulations and the proofs of \cref{thm:main filter quotient,thm:main model}. This necessitates introducing (and reviewing) precise definitions and some technical lemmas. First we generalize model categories with $\kT$-shapes.

\begin{definition} \label{def:finitary model}
	A \emph{finitary model category with $\kT$-shapes} is a tuple of data $(\cM,\cV, \kT, \overline{\omega})$, 
	\begin{enumerate}[leftmargin=*]
	 \item A right proper finitary model structure $\cM$ on an elementary topos whose cofibrations are monomorphisms (\cref{def:model category}).
		\item A propositional coherent theory $\kT$ and a model thereof in a coherent category $\cV$, meaning a commutative diagram 
		\[ 
			\begin{tikzcd}
				\kT_1 \arrow[r, "m_1"] \arrow[d] & \Mono(\cV) \arrow[d] \\
				\kT_0 \arrow[r, "m_0"] & \cV 
			\end{tikzcd},
		\]
		such that $m_0$ preserves products and $m_1$ preserves Cartesian morphisms, and fiber-wise preserves join and meet.
		\item A coherent functor $\overline{\omega}\colon\cV\to\cM$, meaning it preserves finite limits, regular epimorphisms and finite unions.
		\item For any object $U \in \cV$, the functor $(\overline{\omega}(U) \times -) \colon \cM \to \cM$ preserves acyclicity of cofibrations.
	\end{enumerate}
\end{definition}

Observe that this definition strictly generalizes \cite[Definition A.13]{riehlshulman2017rezktypes} and we can determine when they coincide.

\begin{lemma} \label{lemma:finite vs. infinite}
	 A finitary model category with $\kT$-shapes  $(\cM,\cV, \kT, \overline{\omega})$	is a model category with $\kT$-shapes if and only if the underlying category of $\cM$ is locally presentable and the model structure is cofibrantly generated.
\end{lemma} 

Having established the proper model categorical definition, we now need a precise definition of a model of simplicial homotopy type theory. One common model of type theories is given by a comprehension category over $\C$ (i.e. a Grothendieck fibration over $\C$ with a Cartesian functor to the target projection $\C^\rightarrow \to \C$). Given the additional structure present in sHoTT, we need to adjust this definition, resulting in a \emph{comprehension category with shapes}. See \cite[Definition A.5]{riehlshulman2017rezktypes} for a precise formulation. 

In particular, right above \cite[Theorem A.16]{riehlshulman2017rezktypes}, the authors assign to each model category with $\kT$-shapes such a comprehension category in a way that makes no use of presentability or cofibrant generation and hence applies in the same manner to finitary model categories with $\kT$-shapes. We call it the \emph{induced comprehension category with shapes}. We are now ready to precisely state and prove \cref{thm:main model}.

\begin{theorem}[Precise formulation of \cref{thm:main model}] \label{thm:main model precise}
 Let $(\cM,\cV, \kT, \overline{\omega})$ be a finitary model category with $\kT$-shapes. Then the induced comprehension category with shapes has pseudo-stable coherent tope logic with type eliminations for tope disjunction, and also pseudo-stable extension types satisfying relative function extensionality.
\end{theorem}

\begin{proof}[Proof of \cref{thm:main model}]
	First we immediately observe that it has pseudo-stable coherent tope logic \cite[Definition A.7]{riehlshulman2017rezktypes}, following the discussion in \cite[Above Theorem A.16]{riehlshulman2017rezktypes}. Next, we want to prove that we have pseudo-stable extension types. Following \cite[Definition A.10]{riehlshulman2017rezktypes}, we only need to prove that for a given diagram a suitable representing object exists. The proof, given in \cite[Theorem A.16]{riehlshulman2017rezktypes}, purely relies on the existence of well-behaved pullbacks and Cartesian closure, meaning the exact same proof is still valid. Finally, to deduce relative function extensionality, we similarly apply the argument in \cite[Theorem A.17]{riehlshulman2017rezktypes}, which only requires the property of a Cartesian closed model structure.
\end{proof}

\begin{remark} \label{rem:strictification}
	The paragraph after \cite[Theorem A.18]{riehlshulman2017rezktypes} explains that we can then use the methods from \cite{lumsdainewarren2015localuniverses} to construct a strict comprehension category with shapes having strictly stable coherent tope logic with type eliminations for tope disjunction and strictly stable extension types satisfying relative function extensionality, giving us a model for simplicial homotopy type theory. This strictification process has been done explicitly in \cite{weinberger2022strict} for a specific class of models, namely categories of simplicial objects in type-theoretic model topoi, in the sense of Shulman \cite{shulman2019inftytoposunivalent}.
\end{remark}
	
We now proceed to the second goal and show that finitary model categories with $\kT$-shapes are compatible with filter quotients. 

\begin{definition} \label{def:compatible filters}
 Let $(\cM,\cV, \kT, \overline{\omega})$ be a finitary model category with $\kT$-shapes. A \emph{model filter for $\kT$-shapes} is a triple $(\Phi_\kT,\Phi_\cV,\Phi_\cM)$, satisfying the following conditions:
	\begin{itemize}[leftmargin=*]
		\item $\Phi_\kT$ is a filter of subterminal objects	on $\kT$.
		\item $\Phi_\cV$ is a filter of subterminal objects on $\cV$.
		\item $\Phi_\cM$ is a model filter on $\cM$.
		\item The functor $\kT_0 \to \cV$ restricts to a functor $\Phi_{\kT} \to \Phi_{\cV}$, or, equivalently, it induces a functor $\kT_{\Phi_\kT} \to \cV_{\Phi_\cV}$.
		\item The functor $\cV \to \cM$ restricts to a functor $\Phi_{\cV} \to \Phi_{\cM}$, or, equivalently, it induces a functor $\cV_{\Phi_\cV} \to \cM_{\Phi_\cM}$.
	\end{itemize}
\end{definition}

We now work our way towards a proof of \cref{thm:main filter quotient}. This requires several technical lemmas. We commence with the following observation, stating the compatibility between filter quotients and Grothendieck fibrations. It relies on an alternative characterization of filter quotients via filtered colimits. See \cite[Theorem 2.11]{rasekh2025filtermodelcat}, \cite[Lemma 3.15]{rasekh2025filterhott} for a more detailed discussion.

\begin{lemma} \label{lemma:Grothendieck fibration}
	Let $P\colon\D \to \C$ be a Grothendieck fibration and $\Phi$ a filter of subterminal objects on $\C$. Then there is an induced Grothendieck fibration $P_\Phi\colon\D_\Phi \to \C_\Phi$, that is functorial in $\C$. 
\end{lemma}

\begin{lemma} \label{lemma:ktov}
	Let $(\cM,\cV, \kT, \overline{\omega})$ be a finitary model category with $\kT$-shapes and $(\Phi_\kT,\Phi_\cV,\Phi_\cM)$ a model filter for $\kT$-shapes. Then $\Phi_{\kT}$ induces a commutative diagram
	\[ 
		\begin{tikzcd}
			(\kT_1)_{\Phi_\kT} \arrow[r, "m_1"] \arrow[d] & \Mono(\cV)_{\Phi_\cV} \arrow[r, "\simeq"] \arrow[d] & \Mono(\cV_{\Phi_\cV}) \arrow[dl] \\
			(\kT_0)_{\Phi_\kT} \arrow[r, "m_0"] & \cV_{\Phi_\cV} 
		\end{tikzcd},
	\]
	such that the bottom functor preserves	products and the top functor preserves the lattice structure in the fibers.
\end{lemma}

\begin{proof}
 The compatibility of $m_0$ follows by definition (\cref{def:compatible filters}), and the associated compatibility of $m_1$ follows from \cref{lemma:Grothendieck fibration}. By \cref{prop:filter quotient category}, products in $\C_{\Phi_{\C}}$ are given by products in $\C$ and hence are preserved by $m_0$. We similarly deduce that $m_1$ preserves the lattice structure in the fiber. Finally, a morphism $f$ in $\V_{\Phi_\cV}$ is a monomorphism if there exists a $U$ in $\Phi_{\cV}$ such that $f \times U$ is a monomorphism in $\V$ (as monomorphism are characterized via a pullback condition and \cref{prop:filter quotient category}). This gives us the desired equivalence $\Mono(\cV)_{\Phi_\cV}\simeq \Mono(\cV_{\Phi_\cV})$.
\end{proof}

\begin{lemma} \label{lemma:regular functor}
 Let $F\colon\C \to \D$ be a regular functor between regular categories. Let $\Phi_\C, \Phi_\D$ be a filter of subobject on $\C,\D$, such that $F$	restricts to a functor $\Phi_\C \to \Phi_\D$. Then $F$ induces a functor $\C_{\Phi_\C} \to \D_{\Phi_\D}$, which is regular.
\end{lemma}

\begin{proof}
	We need to prove that $F$ preserves finite limits, epimorphisms and finite unions. We have the following diagram
	\[
		\begin{tikzcd}
			\C \arrow[d, "P_{\Phi_\C}"] \arrow[r, "F"] & \D \arrow[d, "P_{\Phi_\D}"]\\ 
			\C_{\Phi_\C} \arrow[r, "F"] & \D_{\Phi_\D}
		\end{tikzcd}. 
	\] 
	By \cref{prop:filter quotient category}, the functor $P_{\Phi_\C}$ preserves and reflects finite limits, epimorphisms and finite unions. Hence the desired result follows from the fact that $F\colon \C \to \D$ is regular.
\end{proof}

\begin{theorem}[Precise formulation of \cref{thm:main filter quotient}] \label{thm:main filter quotient precise}
	Let $(\cM,\cV, \kT, \overline{\omega})$ be a finitary model category with $\kT$-shapes. Also, let $(\Phi_\kT,\Phi_\cV,\Phi_\cM)$ be a model filter for $\kT$-shapes. Then $(\cM_{\Phi_{\cM}},\cV_{\Phi_{\cV}}, \kT_{\Phi_{\kT}}, \overline{\omega})$ is a finitary model category with $\kT_{\Phi_\kT}$-shapes. Moreover, if $\kT$ has a strict interval, then so does $\kT_{\Phi_{\kT}}$.
\end{theorem}

\begin{proof}[Proof of \cref{thm:main filter quotient}]
	We check the conditions in \cref{def:finitary model}, and finally the conditions of a strict interval. 
	\begin{enumerate}[leftmargin=*]
		\item $\cM_{\Phi_\cM}$ is an elementary topos (\cref{prop:filter quotient category}), right proper (\cref{thm:filter quotient model structure}), and cofibrations are monomorphisms (by combining \cref{prop:filter quotient category,thm:filter quotient model structure}).
		\item As $\Phi_\kT$ is compatible with $\Phi_\cV$, this follows from \cref{lemma:ktov}.
		\item A $\Phi_\cV$ is compatible with $\Phi_\cM$, this follows from \cref{lemma:regular functor}. 
		\item For all $U$, $\overline{\omega}(U) \times -\colon \cM_{\Phi_\cM} \to \cM_{\Phi_\cM}$ preserves acyclic cofibrations, by \cref{thm:filter quotient model structure} and the fact that $\overline{\omega}(U) \times -\colon \cM \to \cM$ preserves acyclic cofibrations.
		\item By construction, a strict interval is the data of an object $\2$ in $\kT_0$ along with two points $0,1$ therein satisfying $8$ axioms given in \cite[3.1]{riehlshulman2017rezktypes}, which in the Grothendieck fibration $\kT_1 \to \kT_0$ are expressed via products, joins and meets. By \cref{prop:filter quotient category}, in the following diagram 
		\[
			\begin{tikzcd}
				\kT_1 \arrow[r, "P_{\Phi_{\kT}}"] \arrow[d] & (\kT_1)_{\Phi_\kT} \arrow[d] \\ 
				\kT_0 \arrow[r, "P_{\Phi_{\kT}}"] & (\kT_0)_{\Phi_{\kT}}		
			\end{tikzcd}
		\] 
		the horizontal functors preserve all these properties, hence $\kT_{\Phi_\kT}$ inherits the strict interval from $\kT$. \qedhere
	\end{enumerate} 
\end{proof}

\begin{remark}
	We would expect that for a model filter for $\kT$-shapes $(\Phi_\kT,\Phi_\cV,\Phi_\cM)$ on simplicial objects in a type-theoretic model topos $\cM^{\DD^{op}}$, we can employ similar methods to the ones in \cite{weinberger2022strict} to deduce that the filter quotient is a strict comprehension category with shapes having strictly stable coherent tope logic with type eliminations for tope disjunction and strictly stable extension types satisfying relative function extensionality, as explained in \cref{rem:strictification}. We will not pursue this further here.
\end{remark}

\section{Proof: A Model	of sHoTT that is not Simplicial} \label{sec:notsimplicial}
In this final section we focus on the proof of \cref{thm:main nonsimplicial}. Throughout this section $\cF$ is a non-principal filter on the set of natural numbers $\bN$. The proof requires some basic observations regarding $\prod_\cF \sset^{\DD^{op}}$. 

\begin{lemma} \label{lemma:unique arrow}
	Let $U$ be a subterminal object in $\Ho_\infty\prod_\cF \sset^{\DD^{op}}$. There is a unique object $A$ in $\prod_\cF \sset^{\DD^{op}}$ with the following properties:
		\begin{enumerate}[leftmargin=*]
			\item $A$ is $0$-truncated.
			\item The $(-1)$-truncation of $A$ is $U$.
			\item $\Map(U,A)$ has $2$ elements. 
			\item The induced map $0 + 1\colon U \coprod U \to A$ is not an equivalence. 
			\item Every non-trivial subobject of $A$ is also a subobject of $U \coprod U$.
			\item The map $0 + 1$ does not have a non-trivial factorization. 
		\end{enumerate}
\end{lemma}

\begin{proof}
	First, observe $\Delta^1 \times U$ has this property. Let $A$ be another object with this property. By assumption $A_0 = \{0,1\} \times U$. Now, by assumption the induced map $0 + 1\colon U \coprod U \to A$ is not an equivalence, meaning $A$ must have additional cells, such as $\sigma \in A_n$. If the boundary of $\sigma$ inside $A_0$ is all $0$ (or $1$), then restricting to the fiber of $A$ over $0$ (or $1$) would be an additional subobject, which is a contradiction. Hence, there must be a non-trivial cell from $0$ to $1$ (or $1$ to $0$), giving us a map $i\colon \Delta^1\times U \to A$. As this is a factorization of $0 +1$, it follows that $i$ is an equivalence. 
\end{proof}

\begin{remark}
	It is an interesting historical remark that this characterization is directly motivated by work in \cite{toen2005unicity}, where a similar axiomatization is used to classify automorphisms of the $\infty$-category of $\infty$-categories.
\end{remark}

\begin{definition}
	An $\infty$-category $\C$ has \emph{suitable propositions} if it satisfies the following conditions:
	\begin{itemize}[leftmargin=*]
		\item $\C$ has a terminal object $1$. 
		\item $\C$ has the coproducts of the terminal object. 
		\item The inclusion $\tau_{-1}\C \to \C$ has a left adjoint. 
	\end{itemize}
\end{definition}

We have the following straightforward observation about $\infty$-categories with suitable propositions.

\begin{lemma} \label{lemma:yoneda}
 Let $\C$ be an $\infty$-category with suitable propositions and denote the terminal object by $1$. 
	\begin{itemize}[leftmargin=*]
		\item $\C$ is tensored over finite sets, given by coproducts of the terminal objects. 
		\item $\C^{\DD^{op}}$ is tensored over finite simplicial sets.
		\item There is a Yoneda embedding $\Y\colon \DD \to \sset^{fin} \to \C^{\DD^{op}}$, that maps $[n]$ to $\Y[n] = \Delta^n \otimes 1$.
	\end{itemize}
	We call $\Y$ the \emph{Yoneda functor}.
\end{lemma}

\begin{definition}
	Let $\C$ be an $\infty$-category with suitable propositions, and denote the Yoneda functor by $\Y$. An object $X$ is called \emph{externally discrete} if it is local with respect to the objects $\tau_{-1}X \times \Y[n]$, meaning for all $n \geq 0$ the unique map $\Y[n] \to 1$ induces equivalences 
	\[\Map_\C(\tau_{-1}X, X) \simeq \Map_\C(\Y[n] \times \tau_{-1}X, X).\]
	We denote the full sub-$\infty$-category of externally	discrete objects by $\C^{\mathrm{disc}}$.
\end{definition}

\begin{example} \label{example:discrete}
	Let $\DD \to \prod_\cF \sset^{\DD^{op}}$ be the evident embedding. Then an object $(X_n)_{n \in \bN}$ is externally discrete if and only if the set of natural numbers $n$ for which $X_n$ is a discrete simplicial set is in $\cF$. This, in particular, includes all $(X_n)_{n \in \bN}$ that are not discrete for finitely many indices. In particular, for an increasing sequence of natural numbers that is not eventually constant $(a_n)_{n \in \bN}$, $(S^{a_n})_{n \in \bN} = (\Delta^{a_n} / \partial \Delta^{a_n})_{n \in \bN}$ is externally discrete.
\end{example}

Let us make the following basic observation about this definition. 

\begin{lemma} \label{lemma:discrete in simplicial} 
	 Let $\C$ be an $\infty$-category with suitable propositions, and denote the Yoneda functor by $\Y$. Then $(\C^{\DD^{op}})^{\mathrm{disc}} \simeq \C$.
\end{lemma} 

\begin{lemma} \label{lemma:retract}
	Let $\C$ be an $\infty$-category with suitable propositions, and denote the Yoneda functor by $\Y$. An object $X$ is externally discrete if and only if it is local with respect to $\Y[1]^n$.
\end{lemma}

\begin{proof}
	On the one side $\Y[1]^n$ is the colimit of $\Y[n]$ along $\Y[n-1]$ \cite[Proof of Lemma 10.3]{rezk2001css}. Hence, externally discrete implies local with respect to $\Y[1]^n$. On the other side $\Y[n]$ is the retract of $\Y[1]^n$. Hence local with respect to $\Y[1]^n$ implies externally discrete.
\end{proof}

\begin{lemma} \label{lemma:coproducts terminal}
	Let $\C$ be an $\infty$-category, such that there is an equivalence $\C^{\DD^{op}} \simeq \prod_\cF \sset^{\DD^{op}}$.
	\begin{itemize}[leftmargin=*]
		\item There is an equivalence $\tau_{-1}(\C^{\DD^{op}}) \simeq \tau_{-1}\C$.
		\item $\C$ has suitable propositions. 
	\end{itemize}
\end{lemma}

\begin{proof}
	$(1)$ We can restrict the equivalence to an equivalence of $(-1)$-truncated objects $\tau_{-1}F\colon\tau_{-1}\prod_\cF \sset^{\DD^{op}} \xrightarrow{ \ \simeq \ } \tau_{-1}\C^{\DD^{op}}$. A $(-1)$-truncated object in $\C^{\DD^{op}}$ is precisely a functor $\DD^{op} \to \tau_{-1}\C$. However, by definition, $\tau_{-1}\C$ is a $0$-category (preorder), and so this corresponds to functors out of the $0$-category associated to $\DD^{op}$, which is simply the terminal category, which means $\tau_{-1}(\C^{\DD^{op}}) \simeq \Fun([0],\tau_{-1}\C) \simeq \tau_{-1}\C$.

	$(2)$ First we show $\tau_{-1}\colon\tau_{-1}\C \to \C$ has a left adjoint. The inclusion $\tau_{-1}(\prod_\cF \sset^{\DD^{op}}) \to \prod_\cF \sset^{\DD^{op}}$ has a left adjoint, hence, via the equivalence, the inclusion $\tau_{-1}\C^{\DD^{op}} \to \C^{\DD^{op}}$ has a left adjoint. Combining this with the equivalence from the previous item, we hence get the left adjoint $\C^{\DD^{op}} \to \tau_{-1}\C$. We can now restrict the domain and get the desired left adjoint $\tau_{-1}\colon\C \to \tau_{-1}\C$.

	The terminal object in $\prod_\cF \sset^{\DD^{op}}$ is $(-1)$-truncated, hence it is also $(-1)$-truncated in $\C^{\DD^{op}}$. By the first item it hence lives in $\C$. Now, the coproduct of discrete objects is discrete, hence the coproduct of the terminal object is also in $\C$.
\end{proof}

\begin{lemma} \label{lemma:discrete objects}
	An equivalence $F\colon\prod_\cF \sset^{\DD^{op}} \xrightarrow{ \ \simeq \ } \C^{\DD^{op}}$ restricts to an equivalence $F\colon(\prod_\cF \sset^{\DD^{op}})^{\mathrm{disc}} \xrightarrow{ \ \simeq \ } \C$.
\end{lemma}

\begin{proof}
	By \cref{lemma:coproducts terminal}, $\C$ has suitable propositions. So, by \cref{lemma:discrete in simplicial}, $(\C^{\DD^{op}})^{\mathrm{disc}} \simeq \C$. Now, by \cref{lemma:unique arrow}, $F$ preserves $U \times \Delta^1$, for all subterminal objects $U$. Also, $F$ preserves products, meaning $F$ also preserves the $n$-fold product $(\Delta^1)^n \times U$. Hence, $F$ preserves and reflects object local with respect to $(\Delta^1)^n \times U$. By \cref{lemma:retract}, those are precisely the externally discrete objects. 
\end{proof}

With these lemmas at hand, we can prove \cref{thm:main nonsimplicial}.

\begin{proof}[Proof of \cref{thm:main nonsimplicial}]
	Let us assume we have an equivalence $\C^{\DD^{op}} \simeq \prod_\cF \sset^{\DD^{op}}$. By definition of $\DD$ we have an adjunction 
	\[
		\begin{tikzcd}
			\C \arrow[r, hookrightarrow, shift left=2] \arrow[r, shift right=2, "\bot", leftarrow] & \C^{\DD^{op}}
		\end{tikzcd},
	\]
	which, by the assumption and \cref{lemma:discrete objects}, induces an adjunction
	\[
		\begin{tikzcd}
			(\prod_\cF \sset^{\DD^{op}})^{\mathrm{disc}} \arrow[r, hookrightarrow, shift left=2] \arrow[r, shift right=2, "\bot", leftarrow] & \prod_\cF \sset^{\DD^{op}}
		\end{tikzcd}.
	\] 
	In order to finish the proof it suffices to observe such an adjunction cannot exist. 
	
Using an argument analogous to \cref{rem:too many simplices}, we observe that the natural numbers in $\prod_\cF \sset^{\DD^{op}}$ consist of sequences $(a_n)_{n \in \bN}$, where $(a_n)_{n \in \bN}, (b_n)_{n \in \bN}$ are in the same equivalence class if $\{n \in \bN | a_n = b_n \} \in \cF$. Now, let $\bN^{nc}$ denote the set of equivalence of classes of sequences that are not constant. Using the notation from \cref{example:discrete}, let $S = \coprod_{a_n \in \bN^{nc}} (S^{a_n})_{n \in \bN}$. Then $S$ is not an externally discrete object. Indeed, if it was externally discrete, there would be an $N \in \bN$, such that for all $n > N$, $(S_n)_1 \cong (S_n)_0$. Let $a_n = 1$ if $n < N + 5$ and $a_n = n$ otherwise. Then 
\[ ((S^{a_n})_{N + 1})_0 = S^{a_{N+1}}_0 = (S^1)_0 \not\cong (S^1)_1 =S^{a_{N+1}}_1 = ((S^{a_n})_{N + 1})_1 \]
meaning this property is not satisfied by $(S^{a_n})_{n \in \bN}$, which means it will also not hold for $S$, via the coproduct inclusion $(S^{a_n})_{n \in \bN} \hookrightarrow (S_n)_{n \in \bN}$. 

Let us assume the right adjoint of $S$ is defined and denote it by $(R_n)_{n \in \bN}$. Then, by assumption, it is an externally discrete object, with a map $R \to S$, such that for all $a_n \in \bN^{nc}$, there is a diagram 
\[
	\begin{tikzcd}
		(S^{a_n})_{n \in \bN} \arrow[r, dashed, hookrightarrow] \arrow[dr, hookrightarrow] & R \arrow[d] \\ 
		& S 
	\end{tikzcd},
\]
such that if we pull back the diagram along the coproduct inclusion $(S^{a_n})_{n \in \bN} \to S$, we obtain the diagram 
\[ 
\begin{tikzcd}
(S^{a_n})_{n \in \bN} \arrow[r, dashed, hookrightarrow] \arrow[dr, equal] & (S^{a_n})_{n \in \bN} \times_S R \arrow[d] \\ 
		& (S^{a_n})_{n \in \bN}
\end{tikzcd}.
\]
It suffices to observe there does not exist a discrete object $R$ with such a property. 

The desired conclusion will immediately follow from the more general observation, that for every externally discrete object $R = (R_n)_{n \in \bN}$, there exists a sequence $(a_n)_{n \in \bN}$ in $\bN^{nc}$, such that all maps $(S^{a_n})_{n \in \bN} \to (R_n)_{n \in \bN}$ factor through the terminal object. Indeed, in that case the post-composition $(S^{a_n})_{n\in \bN} \to S$ also factors through the point, and cannot be equivalent to the coproduct inclusion.

Let $R = (R_n)_{n \in \bN}$ be externally discrete. Let us now construct such a sequence. For every $n \in \bN$, if $R_n$ is a discrete simplicial set, let $d_n = n$. Otherwise, let $d_n \in \bN$ be the unique natural number with the property that $R_{n0} \xrightarrow{ \ \simeq \ } R_{nm}$ for all $m \leq d_n$, meaning the maximum natural number for which this isomorphism holds. By the external discreteness assumption on $R$, the sequence needs to converge to $\infty$ as $n$ goes to $\infty$. This means the equivalence class of the sequence $d_n$ is an element in $\bN^{nc}$. Let $d_n-1$ denote the predecessor for $d_n$, which exists as $d_n$ is not $0$. Then, by the assumption on $d_n$, every map $(S^{d_n-1})_{n \in \bN} \to (R_n)_{n \in \bN}$ factors through the point, finishing the proof.
\end{proof}

\bibliographystyle{alpha}
\bibliography{main}

\end{document}